\documentclass{amsproc}      
\usepackage{amsfonts,amssymb,mathrsfs} 
\usepackage{graphicx}        
\usepackage{amsmath}         
\usepackage{epsfig}
\usepackage{color}
\usepackage{tikz,caption,subcaption}
\usepackage{mathtools}
\usepackage{enumerate}
\DeclareMathAlphabet{\mathantt}{OT1}{antt}{li}{it}
\DeclareMathAlphabet{\mathpzc}{OT1}{pzc}{m}{it}
\newtheorem{theorem}{Theorem}[section]
\newtheorem{lemma}[theorem]{Lemma}
\newtheorem{corollary}[theorem]{Corollary}

\theoremstyle{definition}
\newtheorem{definition}[theorem]{Definition}

\theoremstyle{remark}
\newtheorem{remark}[theorem]{Remark}

\newcommand{\N}{\mathbb{N}}              
\newcommand{\R}{\mathbb{R}}              
\newcommand{\Z}{\mathbb{Z}}              
\newcommand{\hd}{\mathcal{H}}            
\newcommand{\B}{\mathcal{B}}             
\newcommand{\C}{\mathcal{C}}             
\newcommand{\de}{\delta}                 
\newcommand{\De}{\Delta}                 
\newcommand{\FO}{\mathcal{F}}            
\newcommand{\A}{\mathfrak{A}}            
\newcommand{\ldq}{\textquotedblleft}     
\newcommand{\rdq}{\textquotedblright}    
\newcommand{\E}{\mathcal{E}}             
\newcommand{\f}{\mathbf{f}}              
\newcommand{\n}{\mathcal{N}}			
\newcommand{\s}{\mathcal{S}}             

\newcommand{\la}{\lambda}                
\newcommand{\La}{\Lambda}
\newcommand{\dom}{\text{dom}}            
\newcommand{\e}{\epsilon}                
\newcommand{\x}{\mathbf{x}}				 
\newcommand{\y}{\mathbf{y}}				 

\newcommand{\al}{\alpha}

\newcommand{\noi}{\noindent}

\newcommand{\uline}{\underline}
\newcommand{\oline}{\overline}

\numberwithin{equation}{section}

\begin{document}
\title[Box dimension of graph of harmonic functions]{On the box dimension of graph of harmonic functions on the Sierpi\'nski gasket}

\author[A. Sahu]{Abhilash Sahu}
\address{Department of Mathematics, Indian Institute of Technology Delhi, Hauz Khas, New Delhi 110016, India}
\email{sahu.abhilash16@gmail.com}
\author[A. Priyadarshi]{Amit Priyadarshi}
\address{Department of Mathematics, Indian Institute of Technology Delhi, Hauz Khas, New Delhi 110016, India}
\email{priyadarshi@maths.iitd.ac.in}

\subjclass[2010]{Primary 28A80}

\keywords{Sierpi\'nski gasket, Laplacian, Harmonic functions, Box dimension, fractal functions }
\date{}

\begin{abstract}
In this paper, we have obtained bounds for the box dimension of graph of harmonic function on the Sierpi\'nski gasket. Also we get upper and lower bounds for the box dimension of graph of functions that belongs to $\text{dom}(\mathcal{E}),$ that is, all finite energy functionals on the Sierpi\'nski gasket. Further, we show the existence of fractal functions in the function space $\dom(\E)$ with the help of fractal interpolation functions. Moreover, we provide bounds for the box dimension of some functions that belong to the family of continuous functions and arise as fractal interpolation functions.
\end{abstract}
\maketitle

\section{Introduction}
A range of fractals which we come across during studies occur as graph of functions. Certainly, many natural phenomena such as wind speeds, solar radiation, population, stock market price etc. are plotted against time, that is, we capture them in graphs. The graph of functions and its box and Hausdorff dimension is of qualitative interest for many authors since past few decades. In particular, Weierstrass type functions fascinate many of them. In 1872, Karl Weierstrass produced a function which is continuous everywhere but differentiable nowhere. Classical Weierstrass function $W_{\la,b} : [0,1)\to \R$ with parameter $b\in \N$ and $\la \in (0,1)$ is defined by
\begin{equation}\label{W1}
W_{\la,b}(x) = \sum_{n=0}^{\infty} \la^n \cos (2\pi b^n x).
\end{equation}

Now we mention some known results of the dimension of Weierstrass type functions. In \cite{KMY} it was proved that if $b\la >1$ then the box dimension of graph of $W_{\la,b}$ is equal to $D = 2 + \frac{\log \la}{\log b}.$ Hunt \cite{hunt} proved that if each $\theta_n$ is chosen independently with respect to the uniform probability in $[0,1],$ then with probability one the Hausdorff dimension of graph of $W_\Theta$ defined by
\begin{equation}\label{W2}
W_{\Theta}(x) = \sum_{n=0}^{\infty} \la^n \cos (2\pi( b^n x+\theta_n))
\end{equation}
is $D = 2 + \frac{\log \la}{\log b}.$ In \cite{KF} chapter 11, it has been shown that for Weierstrass function defined by
\begin{equation}\label{W3}
f(x) = \sum_{n=1}^{\infty} \la^{(s-2)n} \sin (\la^n x)
\end{equation}
for $\la>1$ and $1<s<2,$ the box dimension of graph of $f$ is $s$ provided $\la$ is large enough. Similarly, it was proved that two dimensional Weierstrass function defined by
\begin{equation}
\phi(x,y) = \sum_{n=1}^{\infty} \la^{(s-3)n} \sin (\la^n x)\cos (\la^n y)
\end{equation}
has the box dimension $s$ whenever $\la$ is sufficiently large and $s$ is between 2 and 3. In \cite{bara1} Bara\'nski studied the Weierstrass type functions
\begin{equation}\label{W4}
f(x) = \sum_{n=0}^{\infty} \la_n \phi (b_n x+ \theta_n)
\end{equation}
where $\la_n, b_n>0, \theta_n\in \R$ and $\phi : \R \to \R$ is a
non constant, $\Z$-periodic and Lipschitz function. He proved that if a function $f$ of the form \eqref{W4} satisfies $\frac{\la_{n+1}}{\la_n} \to 0$ and $\frac{b_{n+1}}{b_n} \to \infty$ as $n \to \infty$ then
$$\dim_H(graph f) = \uline{\dim}_B(graph f) = 1 + \liminf_{n \to \infty} \frac{\log^+ d_n}{\log(b_{n+1}d_n/d_{n+1})}$$ and
$$ \oline{\dim}_B(graph f) = 1 + \limsup_{n \to \infty} \frac{\log^+ d_n}{\log b_n},$$
where $ \log ^+ = \sup\{\log,0\}$ and $d_n = \la_1b_1+\dots+\la_n b_n.$
Shen \cite{shen} proved that for any function of the following type
\begin{equation}\label{W5}
f_{\la,b}^\phi(x) = \sum_{n=0}^{\infty} \la^n \phi (b^n x)
\end{equation}
where $\phi : \R \to\R$ is a $\Z$-periodic, non constant, $\C^2$-function and $b\geq 2,$ then there exist a constant $K_0$ depending on $\phi$ and $b$ such that if $1<\la b<K_0$ then the graph of $f_{\la,b}^\phi$ has Hausdorff dimension $D = 2 + \frac{\log \la}{\log b}.$ Bara\'nski et al. \cite{BBR} studied functions of the type \eqref{W1} for integer $b\geq 2$ and $1/b < \la <1.$ They established that for every $b,$ there exists $\la_b \in (1/b,1)$ such that the Hausdorff dimension of graph of $W_{\la,b}$ is equal to $D = 2 + \frac{\log \la}{\log b}$ for all $\la \in (\la_b,1).$ A simpler proof of this and some other results was given by Keller in \cite{kell}. Xie and Zhou \cite{XZ} constructed a wide range of Weierstrass type functions whose graph attains the box dimension two. Mauldin and Williams \cite{MW} studied the following function
$$W_b(x) = \sum_{n=\infty}^{\infty} b^{-\alpha n} [\phi(b^nx+\theta_n) - \phi(\theta_n)]$$
where $b>1, 0 <\alpha <1,\theta_n$ is an arbitrary real number and $\phi$ is a periodic function with period one. They have shown that there exist a constant $C>0$ such that the Hausdorff dimension of graph of $W_b$ is bounded below by $2-\al- (C/ \ln b)$ when $b$ is large enough. For more details about dimension of Weierstrass functions, readers are encouraged to study the above references and the references therein. Our aim is not to study the dimensions of graph of Weierstrass type functions but to study the graph of functions on the Sierpi\'nski gasket.

Now we discuss briefly about fractal interpolation functions(FIFs) and its box dimension on the real line. Fractal interpolation functions were introduced and studied by Barnsley and his co-researchers in \cite{barn1,barn2,barn3}. Navascu\'es and her co-researchers studied fractal operators and their properties on some suitable function spaces in \cite{MN1,MN2}. Bedford \cite{bed} has shown that fractal interpolation function constructed using linear affinities having H\"older exponent($h$) is related to the box dimension $D$ of the graph by $h\leq 2-D \leq h_\la.$ H\"older exponent of $f$ at $x$ is defined as$$h_x \coloneqq \sup\{\alpha : |f(x)-f(y)| \leq |x-y|^\alpha \text{ for all y in some neighbourhood of }  x\}$$
and H\"older exponent of $f$ is defined as $h \coloneqq \inf \{h_x : x\in I\}.$ If there exists a number $h_\la$ such that $h_x=h_\la$ for Lebesgue almost all $x \in I$ then it is called almost everywhere H\"older exponent of $f.$ In \cite{dalla} Dalla et al. obtained the box dimension of some non-affine fractal interpolation functions and also, they have explained this with some explicit examples.
Hardin and Massopust \cite{HM} shown that when interpolation points are not collinear and $\sum_{k=1}^n|a_k| >1$ then  $C(G) = 1 + \log_N(\sum_{k=1}^n|a_k|);$ otherwise $C(G) = 1$ where $C(G),$ the capacity dimension of $G$ is another name for the box dimension. Taking forward the above results Nasim et al. obtained that box dimension of $\alpha$-fractal functions with constant scaling factors in \cite{NGM1} and with variable scaling factors in \cite{NGM2}. Recently, Barnsley and Massopust \cite{BM} studied bilinear fractal interpolation functions as a fixed point of Read-Bajraktarevi\'c operator and presented the box dimension formula of bilinear FIFs.

Afterwards, many researchers have extended the version of fractal interpolation functions to fractal interpolation surfaces(FISs) and also studied the box dimension of these FISs. In 2006, Bouboulis et al. \cite{BDD} constructed recurrent bivariate fractal interpolation surfaces and computed the box dimension for some particular cases. Feng and Sun \cite{FS} studied the box dimension of fractal interpolation surfaces derived from FIFs. They introduced FISs on rectangular domain with arbitrary interpolation nodes and calculated the box dimension by considering its relation with variation of functions.

Later, the concept of fractal interpolation functions was protracted on the Sierpi\'nski gasket by \c{C}elik et al. \cite{CKO}. They have shown the existence of a unique extension of function $F : \s^{(n)} \to \R$ to $f : \s \to \R$ which satisfies $$f(L_\omega(x)) = \alpha_\omega f(x) + h_\omega(x)$$ where $h_\omega(x)$ is a harmonic function on Sierpi\'nski gasket. All the notations in the last equation will be described in the next section. Ruan \cite{ruan} extended this work to P.C.F. self similar sets($K$) and established a sufficient conditions for which the linear FIF have finite energy. They also proved that the solution of Dirichlet problem $$-\De_\mu u =f, u|_{\partial K} = 0$$ is a linear FIF on $K$ if $f$ is linear FIF. Ri and Ruan \cite{RR1} established some more properties of FIFs on Sierpi\'nski gasket. They have shown the min-max property of uniform FIFs and provided a sufficient condition such that uniform FIFs have finite energy. At last they explored the result in view of normal derivative and Laplacian of uniform FIFs. In \cite{LR} Li and Ruan proved the energy finiteness of FIFs on P.C.F. self similar fractals. Also, they discussed some results about Laplacian of FIF on Sierpi\'nski gasket and studied Dirichlet problem on Sierpi\'nski gasket.

An important thing to notice in the last few references is that the authors have constructed FIFs on Sierpi\'nski gasket by taking base function to be a harmonic function to prove their results. This is not necessary in general. In this paper we will construct FIFs on the Sierpi\'nski gasket taking arbitrary base function. Motivated from above results we estimate the box dimension of FIFs on the Sierpi\'nski gasket. Later we obtain bounds for the box dimension of graph of harmonic functions on the Sierpi\'nski gasket. To the best of our knowledge no work has been done related to this till now.

This paper is arranged as follows : In section 2 we recall the preliminaries, that is, definition of box dimension, energy functional, harmonic functions, spaces of finite energy functionals($\dom(\E)$) on the Sierpi\'nski gasket etc. Section 3 is dedicated to FIFs on the Sierpi\'nski gasket and supply sufficient conditions on scaling factors such that FIF belongs to $\dom(\E).$ In section 4 we provide bounds for box dimension of graph of FIFs on the Sierpi\'nski gasket. In Section 5 we give bounds for box dimension of graph of harmonic functions on the Sierpi\'nski gasket.
\section{Preliminaries}
In this section we will recall some definitions which we will use later in this paper.
\begin{definition}[Hausdorff Metric]
\noi Let $(X,d)$ be a metric space. For $A \subseteq X$ and $\epsilon > 0$, let $ N_\epsilon(A) = \{x \in X : d(x,A) < \epsilon \}$ where $d(x,A) = \inf\{d(x,a) : a\in A\}.$ Let $\B(X)$ be the collection of nonempty closed and bounded subsets of $X.$ For $A, B \in \B(X)$ we define $$D_{\hd}(A,B) = \inf \{\epsilon >0 : A \subseteq N_\epsilon(B), B \subseteq N_\epsilon(A) \}.$$
Then this defines a metric on $\B(X)$ and is called the Hausdorff Metric.
\end{definition}

\begin{remark}
Let $(X,d)$ be a metric space and  $\B(X)$ be the collection of all nonempty closed and bounded subsets of $X$. Then $(\B(X),D_{\hd})$ is a complete metric space if $(X,d)$ is a complete metric space(\cite{GE}, Theorem 2.5.3).
\end{remark}

\begin{definition}[Hausdorff Measure]
\noi Suppose that $F$ is a subset of $\R^n$ and $s$ is a non-negative real number. For any $\de > 0,$ we define $$ \hd^s_\de(F) =\inf \left\{\sum_{i=1}^{\infty}|U_i|^s : F \subset \cup_{i=1}^{\infty}{U_i},~0 < |U_i| < \de \right\},$$ where $|U_i|$ denotes the diameter of $U_i$. As $\de$ decreases, $\hd^s_\de(F)$ increases and so approaches a limit(may be $+\infty$) as $\de \to 0^+$. We write $$ \hd^s(F) =  \lim_{\de \to 0^+} \hd^s_\de(F). $$ Then $\hd^s(F)$ is called the $s$-dimensional Hausdorff Measure of $F$.
\end{definition}

\begin{definition}[Hausdorff Dimension]
\noi For any nonempty subset $F$ of $\R^n$, we define the Hausdorff Dimension as
$$ {\dim_H}(F) = \inf \{ s \geq 0 :\hd^s(F) = 0 \} = \sup \{s \geq 0 : \hd^s(F)= \infty\}$$
\end{definition}

\begin{definition}[Box Dimension]
\noi Let $F$ be a nonempty subset of $\R^n$ and $N_\de(F)$ denote the least number of sets of diameter less than or equal to $\de$ which covers $F$.\\
The lower box dimension (box-counting dimension) of $F$ is defined as
$$ \underline{\dim}_B(F) = \liminf_{\de \to 0^+}\frac{\log N_\de(F)}{- \log \de} $$
and the upper box dimension (box-counting dimension) of $F$ is defined as
$$ \overline{\dim}_B(F) = \limsup_{\de \to 0^+}\frac{\log N_\de(F)}{- \log \de}.$$
When these two values are equal, we call the common value as the box dimension of $F$.
\end{definition}

\begin{remark}
For any subset $F$ of $\R^n$, the following holds true
$${\dim_H}(F) \leq \underline{\dim}_B(F) \leq \overline{\dim}_B(F). $$
\end{remark}

Now, we will discuss about the Sierpi\'nski gasket and energy functional in the space of continuous real valued functions in it. Let $\s_0 = \{q_1, q_2, q_3\}$ be three points on $\R^2$ equidistant from each other. Let $L_i(x) = \frac{1}{2}(x-q_i) + q_i$ for $i= 1,2,3$ and $L: \B(\R^2)\to\B(\R^2)$ defined as $L(A) = \cup_{i=1}^3 L_i(A).$ It is well known that $L$ has a unique fixed point $\s$(see, for instance, \cite[Theorem 9.1]{KF}), which is called the Sierpi\'nski gasket. Another way to view the same is $\s = \oline{\cup_{j \geq 0}L^{j}(\s_0)},$ where $L^j$ means $L$ composed with itself $j$ times. We know that $\s$ is a compact set in $\R^2.$ It is well known that the Hausdorff dimension of $\s$ is $\frac{\ln 3}{\ln 2}$ and the $\frac{\ln 3}{\ln 2}$-dimensional Hausdorff measure is finite and nonzero (i.e., $0<\hd^{\frac{\ln 3}{\ln 2}}(\s)<\infty)$ (see, \cite[Theorem 9.3]{KF}). Throughout this paper, we will use this measure and denote it by $\mu$. If $f$ is a measurable function on $\s$, then
$$\|f\|_\infty \coloneqq \inf \{a \in \R : \mu\{x \in \s : |f(x)| > a\} = 0\}.$$
Now, we will define energy functional on the space of continuous functions on the Sierpi\'nski gasket($ \C(\s)$) as follows
The $m^{\text{th}}$ level Sierpi\'nski gasket is $\s^{(m)} \coloneqq \cup_{j=0}^m L^j(\s_0).$ If $x$ and $y$ belongs to same cell of $\s^{(m)}$ we denote it by $x\thicksim_m y.$ We define the $m^{\text{th}}$ level crude energy as
$$E^{(m)}(u) = \sum_{x\thicksim_my} |u(x)-u(y)|^2$$ and the $m^{\text{th}}$ level renormalized energy is
given by $$\E^{(m)}(u) = \left(\frac{5}{3}\right)^m E^{(m)}(u)$$ where $\frac{5}{3}$ is the unique renormalizing factor. Now we can observe that $\E^{(m)}(u)$ is a monotonically increasing function of $m$ because of renormalization. So we define the energy function as
$$\E(u) = \lim\limits_{m \to \infty} \E^{(m)}(u) $$ which exist for all $u$ as an extended real number. Now we define $\dom(\E)$ as the space of continuous
functions $u$ satisfying $\E(u) < \infty.$ In \cite{RS}, it is shown that $\dom(\E)$ modulo constant functions forms a Banach space endowed with the norm
$\|\cdot\|_{\E}$ defined as $$\|u\|_{\E} = \sqrt{\E(u)}.$$
The space $\dom_0(\E)$ is a subspace of $\dom(\E)$ containing all functions which vanishes at boundary of the Sierpi\'nski gasket. For more details see \cite{RS,falc1,kiga1} and references there in.

\begin{definition}[Harmonic function]
A function $f : \s \to \R$ is said to be a harmonic function if $\E^{(m+1)}(f) = \E^{(m)}(f)$ for every
$m\geq 0.$
\end{definition}
\begin{definition}[Piecewise harmonic function]
A function $p : \s \to \R$ is said to be a piecewise harmonic function if there exist a finite partition of $\s$ such that each set is a subset of $\s$ and it is also a Sierpinski gasket in its own right. $p$ restricted to each subset of the partition is a harmonic function. 
\end{definition}
Given three real numbers $a,b,c$ there exist a unique harmonic function $f$ satisfying $f(q_1)=a, f(q_2)=b$ and $f(q_3)=c.$ Function value at intermediate nodes of Sierpi\'nski gasket can be determined by \ldq$\frac{1}{5}-\frac{2}{5}"$ rule
$$f(q_{\omega ij})= \frac{2}{5}h(q_{\omega i}) + \frac{2}{5}h(q_{\omega j}) + \frac{1}{5}h(q_{\omega k})$$
where $\omega \in \Sigma^*$ and $\{i,j,k\}$ are permutation of $\{1,2,3\}.$ We define $\Sigma^* \coloneqq \cup_{m\geq 0}\Sigma^{(m)}$ and  $\Sigma^{(m)}$ is the collection of all words of length $m$ which are possible combinations of symbols 1,2 and 3. We define, $q_{\omega i} \coloneqq L_\omega(q_i)$ for $\omega \in \Sigma^*$ and $i \in \{1,2,3\}.$ For more details about harmonic functions and \ldq$\frac{1}{5}-\frac{2}{5}$\rdq rule see section 1.3 of \cite{RS}.

\section{Fractal operator on $\dom {(\E)}$ and its properties}
Now the question arises : Is there any fractal function in the space $\dom(\E)?$ And the answer is affirmative. In this section we will construct an Iterated Function System(IFS) whose fixed point is graph of a function. We will show that this function belongs to $\dom(\E)$ with some restrictions on the independent parameter. Before proceeding, we recall some definitions.

\begin{definition}[\textbf{Iterated Function System}]
\noi Let $(X,d)$ be a metric space. Let $f_n : X \to X$ for $n \in \La$(a finite index set) be contraction mappings. A finite family of contractions $\{X; f_n : n \in \La\}$ is called an Iterated Function System(IFS).\\
We define $F : \B(X) \to \B(X)$ by $F(A) = \cup_{n\in\La} f_n(A)$ where $\B(X)$ is the collection of all nonempty compact subsets of $X$, $A \in \B(X)$ and $f_n(A) = \{f_n(x) : x \in A\}$.\\
\noi A non empty set $ \A \subset X$ is called an invariant set(attractor) for the IFS $\{X; f_n: n \in \La\}$, if $ \A = \cup_{n\in \La} f_n(\A).$
\end{definition}
\subsection{Fractal Interpolation Functions on the Sierpi\'nski gasket }
\noi Let $\{(\mathbf{x}_i,y_i) \in \s \times \R : i \in \La, \La \text{ is a finite index set}\}$ be given set of data points, where
$\s $ is the Sierpi\'nski gasket. We want to construct a continuous function $f : \s \to \R$ which interpolate the given data
\begin{equation} \label{cond}
 f(\mathbf{x}_i) = y_i, ~~ i \in \La
\end{equation}
 and whose graph $G= \{(\mathbf{x},f(\mathbf{x})) : \mathbf{x} \in \s\}$ is the  attractor of an IFS.

\begin{definition}[\textbf{Fractal Interpolation Function}]
Let $K=\s\times [a,b],$ where interval $[a,b]$ is chosen such a way that each $y_i \in [a,b].$  If there is a  collection of contraction mappings $f_n : K \to K$ such that the unique attractor of the IFS $\{K;f_n : n \in \La, \La \text{ is a finite index set} \}$ is graph of a function on the Sierpi\'nski gasket and the function satisfies \eqref{cond}, then we will call such a function a fractal interpolation function.
\end{definition}

\subsection{Construction of $\alpha$-fractal functions on the Sierpi\'nski gasket}
Let $f$ be a continuous function on the Sierpi\'nski gasket and $n \in \N$ be a fixed number. Suppose the interpolation points are $\{(q_\omega, f(q_\omega)) : \omega \in \Sigma^{(n)}\}.$ Then for fixed $\alpha= \{\alpha_\omega \in (-1,1) : \omega\in \Sigma^{(n)}\}$ we will construct an IFS such that the attractor is graph of a function and passes through the above interpolation points.
We define a function $L_\omega : \R^2 \to \R^2$ by
\begin{equation}\label{eq_3}
L_\omega \coloneqq L_{w_1}\circ L_{w_2}\circ L_{w_3}\circ...\circ L_{w_n} \text{ where } \omega \in \Sigma^{(n)}
\end{equation}
The function $L_\omega$ satisfies following conditions
$$L_\omega(q_1)=q_{\omega1}, L_\omega(q_2)=q_{\omega2}, L_\omega(q_3)=q_{\omega3}$$
and $$\|L_\omega(c)-L_\omega(d)\| \leq \frac{1}{2^{|\omega|}}\|c-d\|.$$
Further, we define a real valued continuous function $F_{\omega} : \s \times \R \to \R$ by
\begin{equation}\label{eq_4}
F_\omega(\x,y) = \alpha_\omega y + f(L_\omega(\x)) - \alpha_\omega~ b(\x)
\end{equation}
where $b$ is a continuous function on the Sierpi\'nski gasket satisfying conditions
$b(q_1) = f(q_1), b(q_2) = f(q_2)$ and $b(q_3) = f(q_3).$
The function $F_\omega$ satisfies following conditions
$$F_{3\tilde{\omega}}(q_1,y_1) = F_{1\tilde{\omega}}(q_3,y_3), F_{2\tilde{\omega}}(q_3,y_3) = F_{3\tilde{\omega}}(q_2,y_2) \text{~and~}  F_{2\tilde{\omega}}(q_1,y_1)= F_{1\tilde{\omega}}(q_2,y_2)$$ for each $\tilde{\omega} \in \Sigma^{(n-1)}$
and
$$\|F_\omega (c,d_1)- F_\omega (c,d_2)\| \leq |\alpha_\omega|\|d_1-d_2\|.$$
Now we define the IFS using the above equations
\begin{equation}\label{eq_5}
\text{IFS} \{K;\f_\omega : \omega \in \Sigma^{(n)}\}~~ \text{where}~~ \f_\omega(\x,y) = (L_\omega(\x), F_\omega(\x,y)).
\end{equation}
This is a contractive IFS. Hence has an unique attractor, let say $G.$ The function corresponding to graph $G$ is named $f^\alpha$ which passes through the interpolation points $\{(q_\omega, f(q_\omega)) : \omega \in \Sigma^{(n)}\}.$
\begin{definition}[\textbf{{$\alpha$-fractal functions}}]\label{def_1}
Let $f^\alpha$ be the function whose graph is an attractor of IFS defined in \eqref{eq_3} - \eqref{eq_5}. Then we call $f^\alpha$ as the $\alpha$-fractal function associated to $f$ with respect to fixed $n\in \N$ and $\alpha = \{\alpha_\omega \in (-1,1): \omega \in \Sigma^{(n)}\}.$
\end{definition}
The above function $f^\alpha$ satisfies the functional equation
\begin{equation}\label{fix-eq1}
  f^\alpha(\x) = f(\x) + \alpha_\omega(f^\alpha - b)\circ L_\omega^{-1}(\x),~~ \forall~~ \x \in L_\omega(\s).
\end{equation}
In the above construction if we take $b = T(f)$ where $T : \C(\s) \to \C(\s)$ is a bounded linear operator satisfying $T(f)(q_1) = f(q_1), T(f)(q_2) = f(q_2)$ and $T(f)(q_3) = f(q_3)$ then $f^\alpha$ satisfies the functional equation
\begin{equation}\label{fix-eq2}
  f^\alpha(\x) = f(\x) + \alpha_\omega(f^\alpha - T(f))\circ L_\omega^{-1}(\x),~~ \forall~~ \x \in L_\omega(\s).
\end{equation}

\begin{definition}[\textbf{$\alpha$-fractal operator}]
We define the $\alpha$-fractal operator $\FO^{\alpha} = \FO^\alpha_{n,T}$ on $\C(\s)$ with respect to fixed $n, \alpha$ and $T$ as
$$\FO^\alpha(f) = f^\alpha $$ where $f^\alpha$ is defined in definition \ref{def_1}.
\end{definition}

\begin{theorem}
Let $f$ be a function in $\dom(\E)$ and $f^{\alpha}$ be the $\alpha$-fractal function corresponding to $f.$ The the function $f^\alpha$ belongs to $\dom(\E)$ if $\|\alpha\|_\infty \leq \frac{1}{\sqrt{3\times 5^n}}$ where $n$ is the fixed number associated to interpolation points, that is, $\{(q_\omega, f(q_\omega)) : \omega \in \Sigma^{(n)}\}.$
\end{theorem}

\begin{proof}
Let $\|\alpha\|\coloneqq \max\{|\alpha_\omega| : \omega \in \Sigma^{(n)}\}$ and $\|T\| \coloneqq \sup\{|T(x)| : \|x\|\leq 1\}.$ By using functional equation \eqref{fix-eq2} and the fact that $(A+B-C)^2 \leq 3(A^2+B^2+C^2),$ we can deduce the following inequality for all $\x,\y \in L_\omega(\s).$  
\begin{align*}
|f^\alpha(\x)&-f^\alpha(\y)|^2 \\&= |f(\x)-f(\y) + \alpha_\omega(f^\alpha - T(f))\circ L_\omega^{-1}(\x)
-\alpha_\omega(f^\alpha + T(f))\circ L_\omega^{-1}(\y)|^2\\
& \leq 3|f(\x)-f(\y)|^2 + 3\alpha_\omega^2|f^\alpha\circ L_\omega^{-1}(\x)-f^\alpha\circ L_\omega^{-1}(\y)|^2 +3\alpha_\omega^2|T(f)\circ L_\omega^{-1}(\x) - T(f)\circ L_\omega^{-1}(\y)|^2.
\end{align*}
Hence, for $m\geq n$ we can estimate the $m^{th}$ level energy of $f^\alpha$ as
\begin{align*}
\E^{(m)}(f^\alpha)& = \left(\frac{5}{3}\right)^m E^{(m)}(f^\alpha)\\ &= \left(\frac{5}{3}\right)^m \sum_{\x\thicksim_m \y} |f^\alpha(\x)-f^\alpha(\y)|^2\\
&\leq \left(\frac{5}{3}\right)^m\sum_{\x\thicksim_m \y} \left( 3|f(\x)-f(\y)|^2 + 3\alpha_\omega^2|f^\alpha\circ L_\omega^{-1}(\x)-f^\alpha\circ L_\omega^{-1}(\y)|^2 +3\alpha_\omega^2|T(f)\circ L_\omega^{-1}(\x) - T(f)\circ L_\omega^{-1}(\y)|^2\right)\\
& \leq 3\E^{(m)}(f) +  \left( \frac{5}{3}\right)^n 3^{n+1}\|\alpha\|^2\E^{(m-n)}(f^\alpha) + \left( \frac{5}{3}\right)^n 3^{n+1}\|\alpha\|^2 \|T\|\E^{(m-n)}(f).
\end{align*}
This gives,
$$\E^{(m)}(f^\alpha)- \left( \frac{5}{3}\right)^n 3^{n+1}\|\alpha\|^2\E^{(m-n)}(f^\alpha)  \leq 3\E^{(m)}(f) + \left( \frac{5}{3}\right)^n 3^{n+1}\|\alpha\|^2 \|T\|\E^{(m-n)}(f).$$
Then taking limit as $m\to \infty$ we get,
$$\E(f^\alpha)- \left( \frac{5}{3}\right)^n 3^{n+1}\|\alpha\|^2\E(f^\alpha)  \leq 3\E(f) + \left( \frac{5}{3}\right)^n 3^{n+1}\|\alpha\|^2 \|T\|\E(f).$$
This implies,
$$0\leq \E(f^\alpha)\left(1- {5}^n 3\|\alpha\|^2\right)  \leq \E(f)\left(3 + {5}^n 3\|\alpha\|^2 \|T\|\right)$$
if $1-  {5}^n 3\|\alpha\|^2\geq 0,$ that is, $\|\alpha\| \leq \frac{1}{\sqrt{3\times 5^n}}.$
Therefore, $\E(f^\alpha) < \infty.$ This completes the proof.
\end{proof}

\begin{corollary}
If $T$ is a linear bounded operator with respect to uniform norm and $\|\alpha\|_{\infty} \leq \frac{1}{\sqrt{3\times 5^n}}$ then
$\FO^\alpha : \dom(\E) \to \dom(\E)$ is linear and bounded with respect to $\|\cdot\|_{\E}$ norm and
$\|\FO^\alpha\|_{\E} \leq \sqrt{\frac{\left(3 + {5}^n 3\|\alpha\|^2 \|T\|\right)}{\left(1- {5}^n 3\|\alpha\|^2\right)}}.$
\end{corollary}

\begin{corollary}
If $f \in \dom(\E)$ then $\|f^\alpha - f\|_\E \leq \|\alpha\|_\infty^2 3^n \|f^\alpha - Tf\|_\E.$
\end{corollary}
\begin{proof}
  From the functional equation \eqref{fix-eq2} it follows directly.
\end{proof}
\section{Box dimension of graph of fractal functions}
In this section we will obtain the upper and lower bounds for the box dimension of graph of fractal functions that we have discussed in the last section. In this section we fix $q_1 = (0,0), q_2 =(1,0),q_3 = (\frac{1}{2},\frac{\sqrt{3}}{2}),n=1$ and interpolation points are $$\{(q_1,f(q_1)),(q_2,f(q_2)),(q_3,f(q_3)),(q_{12},f(q_{12})),(q_{13},f(q_{13})),(q_{23},f(q_{23}))\}.$$
We construct the IFS $\{K, \f_1,\f_2,\f_3\}.$ We define 
$$L_i (\x) = \frac{1}{2}(\x-q_i) + q_i$$ and
$$F_i(\x,y) = \alpha_i y + f(L_i(\x)) - \alpha_i b(\x).$$ Here $b \in \C(\s)$ and satisfies $b(q_i) = f(q_i)$ for all $i=1,2,3.$
So, $\f_i(\x,y) = (L_i(\x),F_i(\x,y))$ for all $i= 1,2,3.$ The fixed point of this IFS gives graph of a function. We try to estimate the box dimension of this graph under certain conditions.

\begin{theorem}
Let $f$ and $b$ are H\"older continuous functions with exponent $\eta_1, \eta_2$ respectively and the interpolation points are not coplanar. Let $f^\alpha$be the $\alpha$-fractal function corresponding to $f$ and $G = \{(\x, f^\alpha(\x)) : \x \in \s\}$ be the graph of $f^\alpha.$ Let  $\psi = \sum_{i=1}^3 \alpha_i$ and $\eta = \min\{\eta_1, \eta_2\}.$ Then the box dimension of G has following bounds :\\
(I) If $\frac{\psi 2^\eta}{3} \leq 1$, then $\frac{\log 3}{\log 2} \leq{\dim}_B(G) \leq 1-\eta + \frac{\log 3 }{\log 2}.$\\
(II)If $\frac{\psi 2^\eta}{3} > 1$, then $\frac{\log 3}{\log 2} \leq{\dim}_B(G) \leq 1+ \frac{\log \psi}{\log 2}.$
\end{theorem}

\begin{proof}
For calculating the box dimension of $G,$ consider the cover of $G$ as cubes with side length $\frac{1}{2^k}.$ Let $\n(k)$ denote the minimum number of cubes of size $ \frac{1}{2^k} \times \frac{1}{2^k}\times\frac{1}{2^k} $ which covers $G.$ For given $\omega \in \Sigma^k,$ we define $\mathbf{A}(k,\omega)$ as a collection of cubes of size $ \frac{1}{2^k} \times \frac{1}{2^k}\times\frac{1}{2^k}$ which has disjoint interior and we denote  $\n(k,\omega)$ to be the number of cubes in $\mathbf{A}(k,\omega).$ Hence $$\n(k) = \sum_{\omega \in \Sigma^k}\n(k,\omega).$$ On applying function $\f_i, i=1,2,3$ on set $\mathbf{A}(k,\omega)$ we observe that it is contained in $\f_i \circ \f_\omega (\blacksquare)\times \R.$ Here, $\blacksquare$ denotes a square $[0,1]\times[0,1].$ So, $\n(k+1) \leq \sum_{i= 1}^3 \sum_{\omega \in \Sigma^k} \n(k+1,i\omega).$
As $f$ and $b$ are H\" older continuous functions there exists $s_1,s_2\geq 0$ and $\eta_1, \eta_2 \geq 0$ such that
$$|f(L_i(x)) - f(L_i(y)) | \leq \frac{s_1}{2^{(k+1)\eta_1}} $$ and
$$|b(x) - b(y)| \leq \frac{s_2}{2^{k\eta_2}}$$ whenever $x,y \in L_\omega(\blacksquare)$ and $\omega \in \Sigma^k.$ Using \eqref{fix-eq1} we get, $\f_i(\mathbf{A}(k,\omega))$ is contained in cuboid whose base is a square of side length $\frac{1}{2^{k+1}}$ and whose height is $\frac{|\alpha_i|\n(k,\omega)}{2^k} + \frac{s_1}{2^{(k+1)\eta_1}} + \frac{|\alpha_i|s_2}{2^{k\eta_2}}.$

\noi Thus,\begin{align*}
\n(k+1,i\omega) &\leq \left(\frac{|\alpha_i|\n(k,\omega)}{2^k} + \frac{s_1}{2^{(k+1)\eta_1}} + \frac{|\alpha_i|s_2}{2^{k\eta_2}}\right)\times 2^{k+1} +2 \\
&= 2|\alpha_i|\n(k,\omega) + s_1 2^{(k+1)(1-\eta_1)} + |\alpha_i|s_2 2^{k(1-\eta_2)+1} +2.
\end{align*}
Summing over $i$ and $\omega$ we obtain
\begin{align*}
\n(k+1) &=
\sum_{i= 1}^3 \sum_{\omega \in \Sigma^k} \n(k+1,i\omega)\\ &\leq  \sum_{i= 1}^3 \sum_{\omega \in \Sigma^k} 2|\alpha_i|\n(k,\omega) + s_1 2^{(k+1)(1-\eta_1)} + |\alpha_i|s_2 2^{k(1-\eta_2)+1} +2\\
&= \sum_{\omega \in \Sigma^k} 2\psi\n(k,\omega) + 3 s_1 2^{(k+1)(1-\eta_1)} + \psi s_2 2^{k(1-\eta_2)+1} +6\\
&= 2\psi\n(k)+ 3^k 3 s_1 2^{(k+1)(1-\eta_1)} + 3^k \psi s_2 2^{k(1-\eta_2)+1} +3^k 6 \\
&\leq 2\psi\n(k)+ 3^k 3 s_1 2^{(k+1)(1-\eta)} + 3^k \psi s_2 2^{k(1-\eta)+1} +3^k 6\\
&\leq 2\psi\n(k) + 3^k 2^{(k+1)(1-\eta)} \left(3s_1+2\psi s_2+ 6\right)\\
&= 2\psi\n(k) + 3^k 2^{(k+1)(1-\eta)}K,
\end{align*}
where $K = 3s_1+2\psi s_2+ 6.$
Applying the above inequality repeatedly we get,
\begin{align*}
\n(k+1) &\leq 2\psi\n(k) +K 3^k 2^{(k+1)(1-\eta)} \\
&\leq 2\psi \left(2\psi\n(k-1) + K 3^{k-1} 2^{k(1-\eta)}\right) + K 3^k 2^{(k+1)(1-\eta)}\\
&= 2^2\psi^2 \n(k-1) + 2\psi K 3^{k-1} 2^{k(1-\eta)} + K 3^k 2^{(k+1)(1-\eta)}\\
&\leq 2^2\psi^2 \left(2\psi \n(k-2) + 3^{k-2} 2^{(k-2)(1-\eta)} \right)+ 2\psi K 3^{k-1} 2^{k(1-\eta)} + K 3^k 2^{(k+1)(1-\eta)}\\
&= 2^3\psi^3 \n(k-2) + K2^2\psi^2 3^{k-2} 2^{(k-2)(1-\eta)}+ 2\psi K 3^{k-1} 2^{k(1-\eta)} + K 3^k 2^{(k+1)(1-\eta)}\\
&= 2^3\psi^3 \n(k-2) +  K 3^{k}2^{(k+1)(1-\eta)}\left(1+ 2\psi 3^{-1}2^{\eta-1} + \left(2\psi 3^{-1}2^{\eta-1}\right)^2 \right).
\end{align*}
Continuing this process $k$ times we get
\begin{equation}\label{fin}
\n(k+1) \leq 2^{k+1} \psi^{k+1}\n(0) + K3^{k}2^{(k+1)(1-\eta)}\left(1+\psi 3^{-1}2^{\eta} + \left(\psi 3^{-1}2^{\eta}\right)^2 + \cdots + \left(\psi 3^{-1}2^{\eta}\right)^k \right).
\end{equation}
\uline{Case I} Consider that $\frac{\psi 2^\eta}{3} \leq 1.$ Then we have following bounds for \eqref{fin}
\begin{align*}
\n(k+1) &\leq 2^{k+1} \psi^{k+1}\n(0) + K3^{k}2^{(k+1)(1-\eta)}\left(1+\psi 3^{-1}2^{\eta} + \left(\psi 3^{-1}2^{\eta}\right)^2 + \cdots + \left(\psi 3^{-1}2^{\eta}\right)^k \right)\\
& \leq 2^{k+1} \psi^{k+1}\n(0) + K3^{k}2^{(k+1)(1-\eta)}\left(k+1\right)\\
& \leq 2^{k+1} \frac{3^{k+1}}{2^{\eta(k+1)}} \n(0) + K3^{k}2^{(k+1)(1-\eta)}\left(k+1\right)\\
& = 2^{(k+1)(1-\eta)} 3^{k+1} \n(0) + K3^{k}2^{(k+1)(1-\eta)}\left(k+1\right)\\
& \leq 2^{(k+1)(1-\eta)} 3^{k+1} (k+1)\left(\n(0) + \frac{K}{3} \right).
\end{align*}
In the above estimate, third inequality follows from the assumption $\frac{\psi 2^\eta}{3} \leq 1,$ this implies, $\psi^{k+1} \leq \frac{3^{k+1}}{2^{\eta(k+1)}}.$ This gives
\begin{align*}
{\dim}_B(G) &\leq \lim_{k \to \infty} \frac{\log \n(k+1)}{-\log 2^{-(k+1)}} \leq \lim_{k \to \infty} \frac{\log\left({2^{(k+1)(1-\eta)} 3^{k+1} (k+1)\left(\n(0) + K/3 \right)}\right)}{-\log 2^{-(k+1)}}\\
& =\lim_{k \to \infty} \frac{\log{2^{(k+1)(1-\eta)}}}{-\log 2^{-(k+1)}} +\lim_{k \to \infty} \frac{\log {3^{k+1}} }{-\log 2^{-(k+1)}} + \lim_{k \to \infty}\frac{\log(k+1)}{-\log 2^{-(k+1)}} + \lim_{k \to \infty}\frac{\log\left(\n(0) + K/3 \right)}{-\log 2^{-(k+1)}}\\
&= 1-\eta + \frac{\log 3 }{\log 2}.
\end{align*}
\uline{Case II} Consider the other case $\frac{\psi 2^\eta}{3} >1.$ Then equation \eqref{fin} has following estimate
\begin{align*}
\n(k+1) &\leq 2^{k+1} \psi^{k+1}\n(0) + K3^{k}2^{(k+1)(1-\eta)}\left(1+\psi 3^{-1}2^{\eta} + \left(\psi 3^{-1}2^{\eta}\right)^2 + \cdots + \left(\psi 3^{-1}2^{\eta}\right)^k \right)\\
& = 2^{k+1} \psi^{k+1}\n(0) + K3^{k}2^{(k+1)(1-\eta)} \left(\frac{(\psi 3^{-1}2^{\eta})^k-1}{\psi 3^{-1}2^{\eta} -1}\right)\\
&\leq 2^{k+1} \psi^{k+1}\n(0) + K3^{k}2^{(k+1)(1-\eta)} \left(\frac{(\psi 3^{-1}2^{\eta})^k}{\psi 3^{-1}2^{\eta} -1}\right)\\
&= 2^{k+1} \psi^{k+1}\n(0) + K 2^{(k+1-\eta)} \left(\frac{\psi^k}{\psi 3^{-1}2^{\eta} -1}\right)\\
&\leq 2^{k+1} \psi^{k+1}\n(0) + K 2^{(k+1)} \left(\frac{\psi^{k+1}}{\psi 3^{-1}2^{\eta} -1}\right)\\
&= 2^{k+1} \psi^{k+1} \left(\n(0) +  \left(\frac{K}{\psi 3^{-1}2^{\eta} -1}\right)\right).\\
\end{align*}
Hence, we estimate the box dimension as follows
\begin{align*}
{\dim}_B(G) &\leq \lim_{k \to \infty} \frac{\log \n(k+1)}{-\log 2^{-(k+1)}} \leq \lim_{k \to \infty} \frac{\log \left(2^{k+1} \psi^{k+1} \left(\n(0) +  \left(\frac{K}{\psi 3^{-1}2^{\eta} -1}\right)\right)\right)}{-\log 2^{-(k+1)}}\\
&= \lim_{k \to \infty} \frac{\log{2^{k+1}}}{-\log 2^{-(k+1)}} + \lim_{k \to \infty}\frac{\log(\psi^{k+1})}{-\log 2^{-(k+1)}} + \lim_{k \to \infty} \frac{\log\left(\n(0) +  \left(\frac{K}{\psi 3^{-1}2^{\eta} -1}\right)\right)}{-\log 2^{-(k+1)}}\\
&= 1+ \frac{\log \psi}{\log 2}.
\end{align*}
This completes the proof.
\end{proof}

\section{Box dimension of graph of Harmonic functions}
In this section we provide upper and lower bounds for the box dimension of graph of harmonic functions. Later, we will also give upper and lower bounds for box dimension of all functions that belongs to $\dom(\E).$ As we can not find an IFS whose attractor is a graph of harmonic function, so we use the properties of harmonic function to compute its box dimensions.
\begin{figure}
\begin{center}
\includegraphics[height=6cm, width=8cm]{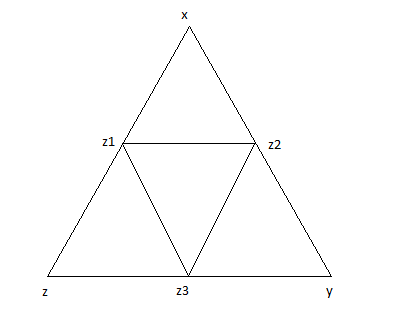}
\caption{Level-1 Sierpi\'nski gasket($\s^{(1)}$) }
  \label{fig-A}
\end{center}
\end{figure}

\begin{lemma}\label{lem1}
Let $h$ be a harmonic function satisfying $|h(\x)-h(\y)| \leq \left(\frac{6}{5}\right)^m\|\x-\y\|$ for each pair of $\x$ and $\y$ in the same cell in the level $\s^{(m)}.$ Then $|h(\tilde{\x})-h(\tilde{\y})| \leq \left(\frac{6}{5}\right)^{(m+1)}\|\tilde{\x}-\tilde{\y}\|$ for each $\tilde{\x}$ and $\tilde{\y}$ that belongs to same cell in the level $\s^{(m+1)}$ and $\tilde{\x}, \tilde{\y}$ belongs to sub cell of the cell containing $\x$ and $\y.$
\end{lemma}
\begin{proof}
We may consider Fig. \ref{fig-A} as one of the cell of m-level Sierpi\'nski gasket and proceed the proof as follow: We calculate the difference for each pair of nodes in the same cell of $\s^{(m+1)}.$ We will use \ldq$\frac{1}{5}-\frac{2}{5}"$ rule and triangle inequality to get the following inequalities.
\begin{equation}
\begin{split}
\left|h(x)-h(z_1)\right| &= \left|h(x)- \frac{2}{5}h(x) - \frac{2}{5}h(z) - \frac{1}{5}h(y)\right|= \left|\frac{3}{5}h(x)-\frac{2}{5}h(z) - \frac{1}{5}h(y)\right| \\
& \leq \frac{2}{5}\left|h(x)-h(z)\right| + \frac{1}{5}\left|h(x)-h(y)\right|\leq \frac{2}{5}\left(\frac{6}{5}\right)^m\|x-z\| + \frac{1}{5}\left(\frac{6}{5}\right)^m\|x-y\| \\
&\leq \frac{3}{5}\left(\frac{6}{5}\right)^m\|x-z\| = \frac{6}{5}\left(\frac{6}{5}\right)^{m}\|x-z_1\|
=\left(\frac{6}{5}\right)^{m+1}\|x-z_1\|.
\end{split}
\end{equation}
\begin{equation}
\begin{split}
\left|h(x)-h(z_2)\right| &= \left|h(x)- \frac{2}{5}h(x) - \frac{2}{5}h(y) - \frac{1}{5}h(z)\right|= \left|\frac{3}{5}h(x)-\frac{2}{5}h(y) - \frac{1}{5}h(z)\right| \\
& \leq \frac{2}{5}\left|h(x)-h(y)\right| + \frac{1}{5}\left|h(x)-h(z)\right|\leq \frac{2}{5}\left(\frac{6}{5}\right)^m\|x-y\| + \frac{1}{5}\left(\frac{6}{5}\right)^m\|x-z\| \\
&\leq \frac{3}{5}\left(\frac{6}{5}\right)^m\|x-y\| = \frac{6}{5}\left(\frac{6}{5}\right)^{m}\|x-z_2\|
=\left(\frac{6}{5}\right)^{m+1}\|x-z_2\|.
\end{split}
\end{equation}
\begin{equation}
\begin{split}
\left|f(z_2)-f(z_1)\right| &= \left|\frac{2}{5}h(x) + \frac{2}{5}h(y) + \frac{1}{5}h(z)- \frac{2}{5}h(x) - \frac{2}{5}h(z) - \frac{1}{5}h(y)\right|= \left|\frac{1}{5}h(y) - \frac{1}{5}h(z)\right| \\
& =\frac{1}{5}\left|h(y)-h(z)\right|\leq \frac{1}{5}\left(\frac{6}{5}\right)^m\|y-z\| = \frac{2}{5}\left(\frac{6}{5}\right)^m\|z_2-z_1\| \leq \left(\frac{6}{5}\right)^{(m+1)}\|z_2-z_1\|.
\end{split}
\end{equation}
\begin{equation}
\begin{split}
\left|h(z)-h(z_1)\right| &= \left|h(z)-\frac{2}{5}h(x)- \frac{2}{5}h(z) - \frac{1}{5}h(y)\right|= \left|\frac{3}{5}h(z)-\frac{2}{5}h(x)- \frac{1}{5}h(y)\right| \\
& \leq \frac{2}{5}\left|h(z)-h(x)\right| + \frac{1}{5}\left|h(z)-h(y)\right|\leq \frac{2}{5}\left(\frac{6}{5}\right)^m\|z-x\| + \frac{1}{5}\left(\frac{6}{5}\right)^m\|z-y\| \\
&\leq \frac{3}{5}\left(\frac{6}{5}\right)^m\|z-x\| = \frac{6}{5}\left(\frac{6}{5}\right)^{m}\|z-z_1\|
=\left(\frac{6}{5}\right)^{m+1}\|z-z_1\|.
\end{split}
\end{equation}
\begin{equation}
\begin{split}
\left|h(z)-h(z_3)\right| &= \left|h(z)-\frac{2}{5}h(z)- \frac{2}{5}h(y) - \frac{1}{5}h(x)\right|= \left|\frac{3}{5}h(z)-\frac{2}{5}h(y)- \frac{1}{5}h(z)\right| \\
& \leq \frac{2}{5}\left|h(z)-h(y)\right| + \frac{1}{5}\left|h(z)-h(x)\right|\leq \frac{2}{5}\left(\frac{6}{5}\right)^m\|z-y\| + \frac{1}{5}\left(\frac{6}{5}\right)^m\|z-x\| \\
&\leq \frac{3}{5}\left(\frac{6}{5}\right)^m\|z-y\| = \frac{6}{5}\left(\frac{6}{5}\right)^{m}\|z-z_3\|
=\left(\frac{6}{5}\right)^{m+1}\|z-z_3\|.
\end{split}
\end{equation}
\begin{equation}
\begin{split}
\left|f(z_3)-f(z_1)\right| &= \left|\frac{2}{5}h(z) + \frac{2}{5}h(y) + \frac{1}{5}h(x)- \frac{2}{5}h(x) - \frac{2}{5}h(z) - \frac{1}{5}h(y)\right|= \left|\frac{1}{5}h(y) - \frac{1}{5}h(z)\right| \\
& =\frac{1}{5}\left|h(y)-h(x)\right|\leq \frac{1}{5}\left(\frac{6}{5}\right)^m\|y-x\| = \frac{2}{5}\left(\frac{6}{5}\right)^m\|z_3-z_1\| \leq \left(\frac{6}{5}\right)^{(m+1)}\|z_3-z_1\|.
\end{split}
\end{equation}
\begin{equation}
\begin{split}
\left|h(y)-h(z_2)\right| &= \left|h(y)-\frac{2}{5}h(y)- \frac{2}{5}h(x) - \frac{1}{5}h(z)\right|= \left|\frac{3}{5}h(y)-\frac{2}{5}h(z)- \frac{1}{5}h(x)\right| \\
& \leq \frac{2}{5}\left|h(y)-h(z)\right| + \frac{1}{5}\left|h(y)-h(x)\right|\leq \frac{2}{5}\left(\frac{6}{5}\right)^m\|y-z\| + \frac{1}{5}\left(\frac{6}{5}\right)^m\|y-x\| \\
&\leq \frac{3}{5}\left(\frac{6}{5}\right)^m\|y-x\| = \frac{6}{5}\left(\frac{6}{5}\right)^{m}\|y-z_2\|
=\left(\frac{6}{5}\right)^{m+1}\|y-z_2\|.
\end{split}
\end{equation}
\begin{equation}
\begin{split}
\left|h(y)-h(z_3)\right| &= \left|h(y)-\frac{2}{5}h(y)- \frac{2}{5}h(z) - \frac{1}{5}h(x)\right|= \left|\frac{3}{5}h(y)-\frac{2}{5}h(z)- \frac{1}{5}h(x)\right| \\
& \leq \frac{2}{5}\left|h(y)-h(z)\right| + \frac{1}{5}\left|h(y)-h(x)\right|\leq \frac{2}{5}\left(\frac{6}{5}\right)^m\|y-z\| + \frac{1}{5}\left(\frac{6}{5}\right)^m\|y-x\| \\
&\leq \frac{3}{5}\left(\frac{6}{5}\right)^m\|y-z\| = \frac{6}{5}\left(\frac{6}{5}\right)^{m}\|y-z_3\|
=\left(\frac{6}{5}\right)^{m+1}\|y-z_3\|.
\end{split}
\end{equation}
\begin{equation}
\begin{split}
\left|f(z_3)-f(z_2)\right| &= \left|\frac{2}{5}h(z) + \frac{2}{5}h(y) + \frac{1}{5}h(x)- \frac{2}{5}h(x) - \frac{2}{5}h(y) - \frac{1}{5}h(z)\right|= \left|\frac{1}{5}h(z) - \frac{1}{5}h(x)\right| \\
& =\frac{1}{5}\left|h(z)-h(x)\right|\leq \frac{1}{5}\left(\frac{6}{5}\right)^m\|z-x\| = \frac{2}{5}\left(\frac{6}{5}\right)^m\|z_3-z_2\| \leq \left(\frac{6}{5}\right)^{(m+1)}\|z_3-z_2\|.
\end{split}
\end{equation}
Hence, the conclusion is true for each pair of nodes in the same cell of $\s^{(m+1)}.$
\end{proof}

\begin{theorem}
Let $h$ be a harmonic function. Then the box dimension of graph of $h$ is less than or equal to $\frac{\log(18/5)}{\log2}\approx 1.8479.$
\end{theorem}
\begin{proof}
Let $h$ be a harmonic function and $G_h$ is the graph of $h.$ We observe that using Lemma \ref{lem1} repeatedly, we get
$$|h({\x})-h({\y})| \leq \left(\frac{6}{5}\right)^{m} \|h\|_\E \|\x-\y\|$$
whenever $\x$ and $\y$ belongs to same cell of the $\s^{(m)}.$ Also, it holds true for every $m\geq 0.$ Hence to cover the graph of function of each cell of $\s^{(m)}$ with cube of side length $(1/2)^{m}$ we need $\left(\frac{6}{5}\right)^{m} \|h\|_\E+2$ number of cubes. To cover graph of $\s^{(m)}$ we need at most $3^{m}\left(\left(\frac{6}{5}\right)^{m} \|h\|_\E+2\right)$ many cubes.
Hence, we get upper bound for the box counting dimension is
\begin{align*}
{\dim}_B(G_h) &\leq \oline\lim_{\de \to 0} \frac{\log N_\delta(G_h)}{-\log \delta} \leq \oline\lim_{m \to \infty} \frac{\log{3^{m}\left(\left(\frac{6}{5}\right)^{m} \|h\|_\E+2\right)}}{-\log 2^{-m}}\\
& = \frac{\log(18/5)}{\log 2}\approx 1.8479.
\end{align*}
\end{proof}
\begin{corollary}
Let $h$ be a piecewise harmonic function on the Sierpi\'nski gasket then the box counting dimension is less than or equal to$\frac{\log(18/5)}{\log 2} \approx 1.8479.$
\end{corollary}

\begin{proof}
Let $h_1, h_2,\ldots,h_{\ell}$ be finite pieces of harmonic functions of $h.$ Then $G_h = \cup_{i=1}^{\ell} G_{h_i}.$ Hence, by finite stability property of box dimension we have, $${\dim}_B(G_h)\leq \oline{\dim}_B(G_h) = \oline{\dim}_B(\cup_{i=1}^{\ell} G_{h_i}) = \max_{i=1}^{\ell}\oline{\dim}_B G_{h_i} \leq \frac{\log(18/5)}{\log 2}.$$ This is true for each harmonic function $h_i$ the upper box dimension is less than or equal to $\frac{\log(18/5)}{\log 2}.$
\end{proof}

\begin{lemma}\label{lem2}
Let $u$ be any function in $\dom(\E)$ then we have
\begin{equation*}
|u(\x) - u(\y)| \leq \left(\frac{3}{5}\right)^{m/2} \sqrt{\E(u)}
\end{equation*}
whenever $x$ and $y$ belong to the same cell of $\s^{(m)}.$
\end{lemma}

\begin{proof}
Let $x,y$ belongs to same cell of $\s^{(m)}.$ Then
\begin{align*}
|u(\x)-u(\y)|^2 \leq \sum_{\x\thicksim_m\y} |u(\x)-u(\y)|^2 &= \left(\frac{3}{5}\right)^m\left(\frac{5}{3}\right)^m \sum_{\x\thicksim_m\y} |u(\x)-u(\y)|^2\\
 &= \left(\frac{3}{5}\right)^m \E^{(m)}(u) \\
 &\leq \left(\frac{3}{5}\right)^m \E(u).
\end{align*}
Hence, $$|u(x)-u(y)| \leq \left(\frac{3}{5}\right)^{m/2} \sqrt{\E(u)}.$$
This completes the proof.
\end{proof}

\begin{lemma}[\cite{KF}, Proposition 2.5]\label{lem3}
Let $F \subset \R^n$ and suppose that $f : F \to \R^m$ is a Lipschitz transformation, that is, there exist $c\geq 0$ such that $|f (x)- f (y)| \leq  c|x- y|$ for every $x, y \in F.$ Then $\uline{\dim}_B f (F) \leq \uline{\dim}_B F$ and
$\oline{\dim}_B f (F) \leq \oline{\dim}_B F.$
\end{lemma}

\begin{theorem}\label{thm1}
Let $f$ be any function in $\dom(\E).$ Then the box dimension of graph of f has lower bound $\frac{\log 3 }{\log 2}$ and upper bound $\frac{\log(108/5)}{2\log 2}\approx 2.21648.$
\end{theorem}

\begin{proof}
Let $f$ be any function in $\dom(\E)$ and $G_f$ is the graph of $f.$
Define $P : G_f = (\x,f(\x)) \to \R^2$ as $P(\x,f(\x)) = \x.$ Clearly, $P$ is a Lipschitz map and $P(G_f) =\s.$ Hence, by Lemma \ref{lem3} we have ${\dim}_B P(G_f) \leq {\dim}_B G_f$ which is same as $\uline{\dim}_B \s \leq \uline{\dim}_B G_f.$ This implies $\dim_B G_f \geq \frac{\log 3}{\log 2}.$

\noi From Lemma \ref{lem2} we have
$$|f(\x)-f(\y)| \leq \left(\frac{3}{5}\right)^{m/2} \sqrt{\E(u)}$$
whenever $\x,\y$ belongs to same cell of $\s^{(m)}$ and it holds true for every $m\geq 0.$ Hence to cover the graph of each cell of $\s^{(m)}$ with cube of side length $2^{-m}$ we need
$\left(\frac{12}{5}\right)^{m/2} \sqrt{\E(u)}+2$ many cubes. To cover graph of $\s^{(m)}$ we need at most  $3^m\left(\left(\frac{12}{5}\right)^{m/2} \sqrt{\E(u)}+2\right)$ many cubes.
Hence, we get upper bound for box counting dimension as
\begin{align*}
{\dim}_B(G_f) &= \oline\lim_{\de \to 0} \frac{\log \n_\delta(G_f)}{-\log \delta} \leq \oline\lim_{m \to \infty} \frac{\log \left(\left(\frac{108}{5}\right)^{m/2} \sqrt{\E(u)}+2\times3^m\right)}{-\log 2^{-m}}\\
& = \frac{\log(108/5)}{2\log 2}= \frac{\log(21.6)}{\log 4} \approx 2.21648.
\end{align*}
\end{proof}
\begin{remark}\label{rem2}
Any constant function $f : \s \to \R$ is harmonic, $f \in \dom (\E)$ and
$$\dim_H(G_f)=\uline{\dim}_B(G_f) = \oline{\dim}_B(G_f) = \dim_B(G_f) = \frac{\log 3}{\log 2}.$$
Hence, we can say that the lower bound in theorem \ref{thm1} is attained.
\end{remark}

\footnotesize
\nocite{}
\bibliographystyle{abbrv}
\bibliography{ref}
\end{document}